\documentclass[11pt]{amsart}

\usepackage[pdftex]{graphicx}
\usepackage{amssymb}
\usepackage{amsmath}
\usepackage{amsfonts}
\usepackage{amsthm}
\usepackage{amssymb}
\usepackage{mathrsfs}
\newtheorem{theorem}{Theorem}[section]

\newtheorem{corollary}[theorem]{Corollary}

\theoremstyle{definition}

\theoremstyle{remark}
\newtheorem{remark}[theorem]{Remark}

\begin{document}
\title{Scalar Curvature and Harmonic Maps to $S^1$}
\author[D. Stern]{Daniel L. Stern}
\address{Department of Mathematics,
University of Toronto, Canada}
\email{dl.stern@utoronto.ca}

\begin{abstract} 
For a harmonic map $u:M^3\to S^1$ on a closed, oriented $3$--manifold, we establish the identity
$$2\pi \int_{\theta\in S^1}\chi(\Sigma_{\theta})\geq \frac{1}{2}\int_{\theta\in S^1}\int_{\Sigma_{\theta}}(|du|^{-2}|Hess(u)|^2+R_M)$$
relating the scalar curvature $R_M$ of $M$ to the average Euler characteristic of the level sets $\Sigma_{\theta}=u^{-1}\{\theta\}$. As our primary application, we extend the Kronheimer--Mrowka characterization of the Thurston norm on $H_2(M;\mathbb{Z})$ in terms of $\|R_M^-\|_{L^2}$ and the harmonic norm to any closed $3$--manifold containing no nonseparating spheres. Additional corollaries include the Bray--Brendle--Neves rigidity theorem for the systolic inequality $(\min R_M)sys_2(M)\leq 8\pi$, and the well--known result of Schoen and Yau that $T^3$ admits no metric of positive scalar curvature.
\end{abstract}

\maketitle

\section{Introduction}

Let $(M^n,g)$ be a closed, oriented Riemannian manifold. Recall that a circle--valued map $u:M\to S^1=\mathbb{R}/\mathbb{Z}$ is \emph{harmonic} if and only if the gradient one--form $h_u:=u^*(d\theta)$ is harmonic in the standard Hodge sense
\begin{equation}\label{hodge.harm}
dh_u=0,\text{ }d^*h_u=0.
\end{equation}
It is easy to see that \eqref{hodge.harm} holds precisely when $u$ minimizes the Dirichlet energy $E(u)=\int_M|du|^2$ in its homotopy class $[u]\in[M:S^1]$, and it follows from elementary Hodge--theoretic considerations that every class in $[M:S^1]$ contains such a minimizer, unique up to a constant rotation.

By Poincar\'{e} duality, every $(n-1)$-homology class $\alpha\in H_{n-1}(M;\mathbb{Z})$ corresponds to a homotopy class $[u]\in [M:S^1]$ of $S^1$--valued maps whose level sets $\Sigma_{\theta}=u^{-1}\{\theta\}$ represent $\alpha$. In particular, just as one can minimize area among integral currents in $\alpha$ to obtain geometrically distinguished representatives (minimal hypersurfaces), one may also minimize energy in the dual homotopy class in $[M:S^1]$ to obtain a geometrically distinguished \emph{(singular) fibration} of $M$ whose fibers $\Sigma_{\theta}$ represent $\alpha$. In view of the key role played by area--minimizing hypersurfaces in the study of scalar curvature since the pioneering work of Schoen and Yau \cite{SY.psc.1, SY.psc.2, SY.pmt}, it is natural to ask whether a careful study of $S^1$--valued harmonic maps may provide a similarly useful link between scalar curvature and topology. In this note, we provide some evidence for a positive answer, in the three-dimensional setting.

The main result of this paper is the following identity for $S^1$--valued harmonic maps on $3$--manifolds, relating the scalar curvature $R_M$ of $M$ to the average Euler characteristic $\chi(\Sigma_{\theta})$ of the level sets $\Sigma_{\theta}=u^{-1}\{\theta\}$. In general, the level set $\Sigma_{\theta}$ may of course have multiple components, and $\chi(\Sigma_{\theta})$ denotes the sum of their Euler characteristics.

\begin{theorem}\label{key.id} Let $(M^3,g)$ be a closed, oriented $3$-manifold, and let $u: M\to S^1$ be a nontrivial harmonic map. Then the level sets $\Sigma_{\theta}=u^{-1}\{\theta\}$ of $u$ satisfy 
\begin{equation}\label{main.id}
2\pi\int_{\theta\in S^1}\chi(\Sigma_{\theta})\geq \frac{1}{2}\int_{\theta\in S^1}\int_{\Sigma_{\theta}}(|du|^{-2}|Hess(u)|^2+R_M).
\end{equation}
\end{theorem}

The proof essentially consists of applying the Schoen--Yau rearrangement trick to the Ricci term in the Bochner identity
$$\Delta |h|=|h|^{-1}(|Dh|^2-|d|h||^2+Ric(h,h))$$
for $h=u^*(d\theta)$, integrating, and applying the coarea formula. Of course, $u$ will have (a measure-zero set of) critical values in general, so in practice we approximate $|h|$ by $(|h|^2+\delta)^{1/2}$ and let $\delta \to 0$; some care must be taken in passing to the limit, but the analysis required is minimal. We describe the relevant computations in Section \ref{compute}. 

Our main application of Theorem \ref{key.id} is the extension of a theorem of Kronheimer and Mrowka (\cite{KM}, Theorem 2) characterizing the Thurston norm on $H_2(M;\mathbb{Z})$ in terms of the harmonic norm and the $L^2$ norm $\|R_M^-\|_{L^2}$ of the negative part $R_M^-:=\min\{0,R_M\}$ of the scalar curvature. In what follows we will always take our manifolds to be connected, without further comment.

For a closed, oriented $3$--manifold, we recall that the \emph{Thurston norm} (or semi--norm) of a homology class $\alpha\in H_2(M;\mathbb{Z})$ is defined by the minimum
\begin{equation}\label{th.def}
\|\alpha\|_{Th}:=\min\{\chi_-(\Sigma)\mid [\Sigma]=\alpha\in H_2(M;\mathbb{Z})\},
\end{equation}
over all embedded surfaces $\Sigma$ representing $\alpha$, of the quantity $\chi_-(\Sigma)$ given by the sum
\begin{equation}
\chi_-(\Sigma)=\max\{0,-\chi(\Sigma_1)\}+\cdots+\max\{0,-\chi(\Sigma_k)\}
\end{equation}
over the connected components $\Sigma_1,\ldots,\Sigma_k$ of $\Sigma$. Thurston introduced this semi--norm in \cite{Th}, in connection with the study of foliations and fibrations of $3$--manifolds over $S^1$. 

When $M^3$ is endowed with a metric $g$, another natural norm on $H_2(M;\mathbb{Z})$ is the \emph{harmonic norm} 
\begin{equation}
\|\alpha\|_{H,g}:=\|h_{\alpha}\|_{L^2},
\end{equation}
given by the $L^2$ norm of the harmonic one--form $h_{\alpha}\in\mathcal{H}^1(M)$ with integral periods dual to $\alpha$. Equivalently, we see that
\begin{equation}\label{harm.char.2}
\|\alpha\|_H=\|du\|_{L^2}
\end{equation}
is the $L^2$ norm of the gradient for the harmonic map $u:M\to S^1=\mathbb{R}/\mathbb{Z}$ whose level sets represent $\alpha$. It is natural to ask how the harmonic norm relates to the Thurston norm on $H_2(M;\mathbb{Z})$. Using the identity \eqref{main.id}, we establish the following relationship.

\begin{theorem}\label{thurst.thm} Let $(M^3,g)$ be a closed, oriented $3$--manifold containing no nonseparating spheres. Then for any class $\alpha\in H_2(M;\mathbb{Z})$, we have
\begin{equation}\label{thurst.bd}
\|\alpha\|_{Th}\leq \frac{1}{4\pi}\|\alpha\|_{H,g}\|R_g^-\|_{L^2},
\end{equation}
where $R_g^-:=\min\{0,R_g\}$ is the negative part of the scalar curvature. If equality holds for some nontrivial class $\alpha\neq 0$, then $(M,g)$ is covered isometrically by a cylinder $\Sigma \times \mathbb{R}$ over a surface $\Sigma^2$ of constant curvature. 
\end{theorem}
\begin{remark} Since $T^3$ contains no nonseparating spheres, the rigidity statement in Theorem \ref{thurst.thm} recovers in this case the well--known fact that any metric on $T^3$ of nonnegative scalar curvature must be flat. This of course follows from the classical results of Schoen and Yau \cite{SY.psc.1} (as well as Dirac operator methods \cite{GL1}), but perhaps the new proof will be of some interest.
\end{remark}

The estimate \eqref{thurst.bd}--in its dual form, bounding the harmonic norm on $H^2(M)$ above by the product of $\|R^-\|_{L^2}$ and the dual Thurston norm--was proved for irreducible $3$-manifolds by Kronheimer and Mrowka in \cite{KM}. The proof in \cite{KM} is very different from ours, first using the Seiberg--Witten equations (and the Weitzenb\"{o}ck formula for the Dirac operator) to bound the harmonic norm of a monopole class in $H^2(M)$ by $\|R_M^-\|_{L^2}$, then building on deep results of Gabai, Eliashberg--Thurston and others to show that the unit ball of the dual Thurston norm lies in the convex hull of the monopole classes when $M$ is irreducible. The rigidity statement was later supplied by Itoh and Yamase in \cite{IY}. We refer the reader to F. Lin's recent paper \cite{Lin} for some interesting refinements and developments on these themes.

As in \cite{KM}, it is not hard to construct a family of metrics on $M$ for which the inequality \eqref{thurst.bd} approaches equality, giving the following geometric characterization of the Thurston norm. We give a detailed discussion of Theorem \ref{thurst.thm} and the following corollary in Section \ref{thurst.sec} below. 

\begin{corollary}\label{thurst.char} For a class $\alpha\in H_2(M;\mathbb{Z})$ on a closed, oriented $3$--manifold with no nonseparating spheres, the Thurston norm is given by the infimum
\begin{equation}
\|\alpha\|_{Th}=\frac{1}{4\pi}\inf\{\|\alpha\|_{H,g}\|R_g^-\|_{L^2}\mid g\in Met(M)\}
\end{equation}
of the product of the harmonic norm and $\|R^-\|_{L^2}$ over the space $Met(M)$ of all Riemannian metrics on $M$.
\end{corollary}

\begin{remark} In \cite{Katz}, G. Katz has also considered the application of harmonic $S^1$--valued maps to the study of the Thurston norm of $3$--manifolds. The results of \cite{Katz} emphasize topological features of the maps, rather than their role as a mediator between topology and geometry.
\end{remark}

While the preceding results show that circle--valued harmonic maps can be used to recover some geometric inequalities previously obtained through Dirac operator methods, Theorem \ref{key.id} may also be used in the proofs of some rigidity theorems related to the area of minimal surfaces in $M$. On a closed, oriented $3$--manifold $(M^3,g)$, define the \emph{homological $2$--systole}
\begin{equation}\label{sys.def}
sys_2(M):=\inf\{Area(\Sigma^2)\mid \Sigma^2\subset M\text{ embedded, }[\Sigma]\neq 0\in H_2(M;\mathbb{Z})\}
\end{equation}
to be the least area among nonseparating surfaces in $M$. In Section \ref{rig.sec}, we observe that the identity \eqref{main.id} yields a short proof of the following rigidity theorem, originally proved by Bray--Brendle--Neves in \cite{BBN} via the analysis of stable minimal surfaces.

\begin{theorem}[cf. \cite{BBN}]\label{bbn.cor}
On a closed, oriented $3$--manifold $(M^3,g)$ with positive scalar curvature $R_M>0$ and nontrivial homology $H_2(M;\mathbb{Z})\neq 0$, we have
\begin{equation}
(\min R_M)sys_2(M)\leq 8\pi,
\end{equation}
with equality only if $M$ is covered isometrically by a cylinder $S^2\times \mathbb{R}$ over a round sphere.
\end{theorem}

While the arguments in \cite{BBN} apply to systolic quantities somewhat finer than $sys_2(M)$, the identity \eqref{main.id} provides a much shorter path to rigidity. Indeed, one generally appealing feature of \eqref{main.id} is the ease with which it leads to rigidity and splitting statements, based on the existence of nonconstant $S^1$--valued maps of vanishing Hessian. Though we restrict our attention in the present note to the applications discussed above, we expect that Theorem \ref{key.id} and variants thereof will find use in the study of other questions related to the scalar curvature of three--dimensional manifolds.

\section*{acknowledgements}

The author is grateful to Otis Chodosh, Yevgeny Liokumovich, and Antoine Song for several valuable conversations related to this work. For their interest and encouragement, he also thanks Robert Haslhofer, Fernando Cod\'{a} Marques, Alex Nabutovsky, Regina Rotman, and Mark Stern. He also takes the opportunity to thank Hugh Bray for introducing him to the study of scalar curvature during his undergraduate years at Duke.

\section{Derivation of the identity}\label{compute}
In this section we describe the computations from which Theorem \ref{key.id} follows. The only ingredients are the Bochner identity for harmonic one--forms, the Gauss equation, the coarea formula, and the Gauss--Bonnet theorem.

\begin{proof}[Proof of Theorem \ref{key.id}] Let $u: M\to S^1=\mathbb{R}/\mathbb{Z}$ be harmonic, so that the gradient one--form $h:=u^*(d\theta)$ is a harmonic form. The standard Bochner identity for $h$ then reads
\begin{equation}
\Delta \frac{1}{2}|h|^2=|Dh|^2+Ric(h,h),
\end{equation}
and setting
$$\varphi_{\delta}:=(|h|^2+\delta)^{1/2}$$
for any $\delta>0$, it easily follows that
\begin{equation}\label{phi.del.comp}
\Delta \varphi_{\delta}=\frac{1}{\varphi_{\delta}}[\frac{1}{2}\Delta |h|^2-\frac{|h|^2}{\varphi_{\delta}^2}|d|h||^2]\geq \frac{1}{\varphi_{\delta}}[|Dh|^2-|d|h||^2+Ric(h,h)].
\end{equation}

Now, along a regular level set $\Sigma$ of $u$, note that $\nu:=\frac{h}{|h|}$ gives the unit normal to $\Sigma$, so--mimicking Schoen and Yau's trick from the minimal hypersurface setting--we can use the traced Gauss equation
$$Ric(\nu,\nu)=\frac{1}{2}(R_M-R_{\Sigma}+H_{\Sigma}^2-|k_{\Sigma}|^2)$$
to rewrite the Ricci term $Ric(h,h)=|h|^2Ric(\nu,\nu)$ in \eqref{phi.del.comp}. Here, $R_M$ and $R_{\Sigma}$ are the scalar curvatures of $M$ and $\Sigma$, respectively, $k_{\Sigma}$ is the second fundamental form of $\Sigma$, and $H_{\Sigma}=tr_{\Sigma}k_{\Sigma}$ is the mean curvature. In particular, recalling that the second fundamental form $k_{\Sigma}$ of $\Sigma$ is given (up to sign) by the restriction
$$k_{\Sigma}=(|h|^{-1}Dh)|_{\Sigma}$$
of the normalized Hessian $|h|^{-1}Dh$ to $T\Sigma$, we see that
$$|h|^2|k_{\Sigma}|^2=|Dh|^2-2|d|h||^2+Dh(\nu,\nu)^2.$$
Moreover, we see that the mean curvature $H_{\Sigma}=tr_{\Sigma}k_{\Sigma}$ is given by
$$|h|H_{\Sigma}=tr_M(Dh)-Dh(\nu,\nu)=-Dh(\nu,\nu)$$
(where in the last equality we have used the fact that $h$ is harmonic), so that
$$|h|^2(H_{\Sigma}^2-|k_{\Sigma}|^2)=2|d|h||^2-|Dh|^2.$$
Putting these identities together, we find that
\begin{equation}\label{ric.comp}
Ric(h,h)=|h|^2Ric(\nu,\nu)=\frac{1}{2}|h|^2(R_M-R_{\Sigma})+\frac{1}{2}(2|d|h||^2-|Dh|^2).
\end{equation}
Substituting \eqref{ric.comp} for the Ricci term in \eqref{phi.del.comp} and writing $Hess(u)=Dh$, we see now that, along regular level sets of $u$,
\begin{equation}\label{phi.del.2}
\Delta \varphi_{\delta}\geq \frac{1}{2\varphi_{\delta}}[|Hess(u)|^2+|du|^2(R_M-R_{\Sigma})].
\end{equation}

Now, let $A\subset S^1$ be an open set containing the set $\mathcal{C}$ of critical values of $u$, and let $B\subset Reg(u)$ be the complementary closed subset of regular values. Integrating \eqref{phi.del.2} over $u^{-1}(B)$ and using the fact that $\int_M\Delta \varphi_{\delta}=0$, it then follows that
\begin{equation}\label{pre.id.1}
\int_{u^{-1}(B)}\frac{1}{2\varphi_{\delta}}[|Hess(u)|^2+|du|^2(R_M-R_{\Sigma})]\leq -\int_{u^{-1}(A)}\Delta \varphi_{\delta},
\end{equation}
Moreover, since
$$\Delta \varphi_{\delta}\geq \frac{1}{\varphi_{\delta}}(|Dh|^2-|d|h||^2+Ric(h,h))\geq -C_M|h|$$
globally on $M$, we see that
$$-\int_{u^{-1}(A)}\Delta \varphi_{\delta}\leq C_M\int_{u^{-1}(A)}|h|=C_M\int_A Area(\Sigma_{\theta}),$$
where we have used the coarea formula in the last inequality. On the other hand, since $|h|>0$ is bounded away from $0$ on $u^{-1}(B)$, we can pass to the limit $\delta \to 0$ in the left--hand side of \eqref{pre.id.1} to conclude that
$$\int_{u^{-1}(B)}\frac{1}{2|du|}[|Hess(u)|^2+|du|^2(R_M-R_{\Sigma})]\leq C\int_A Area(\Sigma_{\theta}).$$
By the coarea formula and the Gauss--Bonnet theorem, we see finally that
\begin{eqnarray*}\int_{u^{-1}(B)}\frac{|du|}{2}\left(\frac{|Hess(u)|^2}{|du|^2}+(R_M-R_{\Sigma})\right)&=&\frac{1}{2}\int_B\int_{\Sigma_{\theta}}(|du|^{-2}|Hess(u)|^2+R_M)\\
&&-\int_B2\pi \chi(\Sigma_{\theta}),
\end{eqnarray*}
so that the preceding estimate becomes
\begin{equation}\label{almost.id}
\frac{1}{2}\int_B\int_{\Sigma_{\theta}}(|du|^{-2}|Hess(u)|^2+R_M)\leq 2\pi \int_B\chi(\Sigma_{\theta})+C\int_A Area(\Sigma_{\theta}).
\end{equation}

Finally, by Sard's theorem, we can take the measure of $A$ arbitrarily small, and since $\theta \mapsto Area(\Sigma_{\theta})$ is integrable over $S^1$ (by the coarea formula), taking $|A|\to 0$ in \eqref{almost.id} yields the desired identity.
\end{proof}

\begin{remark} More generally, if $\psi \in C^{\infty}(M)$ is any smooth function, the same computation gives the identity
\begin{equation}
\frac{1}{2}\int_{\theta\in S^1}\int_{\Sigma_{\theta}}(|du|^{-2}|Hess(u)|^2+(R_M-R_{\Sigma}))\psi^2\leq -2\int_M\psi\langle d\psi, d|du|\rangle
\end{equation}
in arbitrary dimension. Na\"{i}vely, one might hope to wield this estimate in a manner similar to the stability inequality for minimal hypersurfaces to extract more information about the geometry of the fibers $\Sigma_{\theta}=u^{-1}\{\theta\}$ in dimension $n\geq 4$.
\end{remark}

\section{Scalar curvature and the Thurston norm}\label{thurst.sec}

Let $(M^3,g)$ be a closed, oriented $3$--manifold that contains no nonseparating spheres. Given a nontrivial homology class $\alpha\in H_2(M;\mathbb{Z})$, consider the harmonic map $u:M\to S^1=\mathbb{R}/\mathbb{Z}$ whose fibers $\Sigma_{\theta}=u^{-1}\{\theta\}$ lie in $\alpha$. While any given regular fiber $\Sigma_{\theta}$ may have multiple connected components $\Sigma_{\theta}=S_1\cup \cdots \cup S_k$, it is obvious that each component $S_i$ must have nontrivial pairing
$$\langle S_i,*h\rangle=\int_{S_i}|h|>0$$
with the closed $2$--form $*h$ dual to the gradient one--form $h=u^*(d\theta)$. In particular, since $M$ contains no nonseparating spheres, it follows that every component of $\Sigma_{\theta}$ must have nonpositive Euler characteristic, so that by definition \eqref{th.def} of the Thurston norm, we have
\begin{equation}\label{thurst.bd.1}
\|\alpha\|_{Th}\leq -\chi(\Sigma_{\theta})
\end{equation}
for every regular value $\theta\in S^1$.

\begin{proof}[Proof of Theorem \ref{thurst.thm}]

Combining \eqref{thurst.bd.1} with Theorem \ref{key.id}, and recalling that we're taking as our target the circle $S^1=\mathbb{R}/\mathbb{Z}$ of unit length, we then see that
\begin{eqnarray*}
2\pi \|\alpha\|_{Th}&\leq & -2\pi \int_{\theta\in S^1}\chi(\Sigma_{\theta})\\
&\leq & -\frac{1}{2}\int_{\theta\in S^1}\int_{\Sigma_{\theta}}(|du|^{-2}|Hess(u)|^2+R_M)\\
&\leq &-\frac{1}{2}\int_MR_M|du|,
\end{eqnarray*}
with equality only if $Hess(u)\equiv 0$. Applying Cauchy--Schwarz, it follows in particular that
\begin{equation}
\|\alpha\|_{Th}\leq \frac{1}{4\pi}\|du\|_{L^2}\|R_M^-\|_{L^2}=\frac{1}{4\pi}\|\alpha\|_H\|R_M^-\|_{L^2},
\end{equation}
with equality only if $Hess(u)\equiv 0$ and $R_M\equiv -c|du|$ is constant. If equality holds, then since $Hess(u)\equiv 0$, fixing any connected component $S$ of a level set $\Sigma_{\theta}$, it's easy to see that the gradient flow
$$\Phi:S\times \mathbb{R}\to M,\text{ }\frac{\partial\Phi}{\partial t}=\frac{grad(u)}{|grad(u)|}\circ \Phi$$
gives a local isometry, while the constancy of $R_M$ implies that $S$ has constant scalar curvature $R_S\equiv R_M$, completing the proof of Theorem \ref{thurst.thm}.
\end{proof}

To prove Corollary \ref{thurst.char}, we now follow more or less the construction from Lemma 4 of \cite{KM} to exhibit a family of metrics $g_{r,\delta}$ for which $\frac{1}{4\pi}\|\alpha\|_{H,g}\|R_g^-\|_{L^2}$ approaches the Thurston norm $\|\alpha\|_{Th}$ as $r\to\infty$ and $\delta\to 0$.

\begin{proof}[Proof of Corollary \ref{thurst.char}] Fix a nontrivial homology class $\alpha\in H_2(M;\mathbb{Z})$, and let
$$\Sigma=\Sigma_1\cup\cdots\cup \Sigma_k$$
be an embedded representative of $\alpha$ realizing the Thurston norm
$$\|\alpha\|_{Th}=\chi_-(\Sigma),$$
such that none of the $\Sigma_i$ is a sphere. Let $\Sigma_1,\ldots,\Sigma_p$ be the torus components, so that $\Sigma_{p+1},\ldots,\Sigma_k$ have negative Euler characteristic.

Given $\delta>0$, fix an initial metric $g_{1,\delta}$ which coincides on a neighborhood of each $\Sigma_i$ with the cylinder $\Sigma_i\times [0,1]$, where $\Sigma_i$ is endowed with a flat metric of area $\delta>0$ if $1\leq i\leq p$, and for $p<i\leq k$, $\Sigma_i$ is given a metric of constant scalar curvature
$$R_{\Sigma_i}\equiv-2$$
and area
$$Area(\Sigma_i)=-2\pi \chi(\Sigma_i).$$

Now, for $r>>1$, let $g_{r,\delta}$ be a metric which contains about each $\Sigma_i$ a product region $T_{r,i}\cong \Sigma_i\times [0,r]$ and coincides with $g_{1,\delta}$ on the complement $E:=M\setminus\bigcup_{i=1}^k T_{r,i}$. We then see that
\begin{eqnarray*}
\int_M (R_{g_{r,\delta}}^-)^2dvol_{g_{r,\delta}}&=&\int_E (R_{g_{1,\delta}}^-)^2dvol_{g_{1,\delta}}-r\Sigma_{i=p+1}^k4 (2\pi \chi(\Sigma_i))\\
&=&C(\delta)+8\pi r \|\alpha\|_{Th}.
\end{eqnarray*}
At the same time, we can define a map $v^r:M\to \mathbb{R}/\mathbb{Z}$ in the homotopy class dual to $\alpha$ by setting
$$v^r(x,t):=t/r\text{ for }(x,t)\in T_{r,i}\cong \Sigma_i\times [0,r]$$
and $v^r\equiv 1\equiv 0 \mod \mathbb{Z}$ on $E$, and direct computation gives
\begin{eqnarray*}
\int_M |dv^r|_{g_r}^2dvol_{g_{r,\delta}}&=&\Sigma_{i=1}^k\frac{1}{r^2}vol(T_{r,i})\\
&=&\Sigma_{i=1}^p\frac{\delta}{r}-\Sigma_{i=p+1}^k\frac{2\pi \chi(\Sigma_i)}{r}\\
&=&\frac{1}{r}(p\delta+2\pi \|\alpha\|_{Th}).
\end{eqnarray*}
By definition, the harmonic norm $\|\alpha\|_{H,g_{r,\delta}}$ is bounded above by the $L^2$ norm of $dv^r$, so taking the product of the preceding estimates, we see now that
\begin{eqnarray*}
\|\alpha\|_{H,g_{r,\delta}}^2\|R_{g_{r,\delta}}^-\|_{L^2}^2&\leq & (8\pi\|\alpha\|_{Th}+\frac{C(\delta)}{r})(p\delta+2\pi \|\alpha\|_{Th})\\
&=&16 \pi^2\|\alpha\|_{Th}^2+p\delta (8\pi \|\alpha\|_{Th})+\frac{C(\delta)(p\delta+2\pi\|\alpha\|_{Th})}{r}.
\end{eqnarray*}

For any fixed $\delta>0$, taking $r\to\infty$ in the preceding estimate gives
\begin{equation}
\inf\{\|\alpha\|_{H,g}\|R_g^-\|_{L^2}\mid g\in Met(M)\}\leq (16\pi^2\|\alpha\|_{Th}^2+8\pi p\|\alpha\|_{Th}\delta)^{1/2},
\end{equation}
and taking $\delta \to 0$ gives the desired result
$$\inf\{\|\alpha\|_{H,g}\|R_g^-\|_{L^2}\mid g\in Met(M)\}=4\pi \|\alpha\|_{Th}.$$
\end{proof}

\section{Other rigidity results}\label{rig.sec}

The following is another immediate corollary of Theorem \ref{key.id}.

\begin{corollary}\label{obv.rig.cor} Let $u:M^3\to S^1$ be a nonconstant harmonic map from a closed, oriented $3$--manifold $M$ of positive scalar curvature $R_M>0$. Then
\begin{equation}\label{obv.cor}
2\pi \int_{\theta\in S^1}\chi(\Sigma_{\theta})\geq \frac{1}{2}(\min R_M)\int_{\theta\in S^1}Area(\Sigma_{\theta}),
\end{equation}
with equality only if $M$ is covered isometrically by a cylinder $S^2\times \mathbb{R}$ over a round sphere.
\end{corollary}

As in the proof of Theorem \ref{thurst.thm}, the local splitting of $M$ in the case of equality follows from the existence of a nonconstant map $u:M\to S^1$ of vanishing Hessian, while the constancy of the scalar curvature is an easy consequence of equality in the intermediate estimates. 

\begin{proof}[Proof of Theorem \ref{bbn.cor}]

As discussed in the preceding section, we know (by integrating against $*h$) that every connected component of a regular fiber $\Sigma_{\theta}$ is nonseparating, so by definition \eqref{sys.def} of the homological $2$-systole $sys_2(M)$, denoting by $N(\theta)$ the number of components of $\Sigma_{\theta}$, we have
$$Area(\Sigma_{\theta})\geq N(\theta)\cdot sys_2(M).$$
On the other hand, it's also clear that
$$\chi(\Sigma_{\theta})\leq 2 N(\theta),$$
and applying both inequalities in \eqref{obv.cor}, we see that
\begin{equation}
4\pi \int_{\theta\in S^1}N(\theta)\geq \frac{1}{2}(\min R_M)sys_2(M)\int_{\theta\in S^1}N(\theta).
\end{equation}
In particular, it follows that
$$(\min R_M)sys_2(M)\leq 8\pi,$$
with equality only if $M$ is covered by a cylinder $S^2\times \mathbb{R}$ over a round sphere. \end{proof}

\end{document}